\documentclass{siamltex}
\usepackage{lineno,hyperref}
\usepackage{multirow}
\usepackage[english]{babel}
\usepackage{amsmath}
\usepackage{bm}
\usepackage[matha,mathx]{mathabx}
\usepackage{mathrsfs}
\usepackage{amssymb}
\usepackage{algpseudocode} 
\usepackage{algorithm}
\usepackage{graphicx}
\usepackage{color} 
\usepackage{caption}
\usepackage{enumitem} 

\makeatletter
\newcommand*{\rom}[1]{\expandafter\@slowromancap\romannumeral #1@}
\makeatother

\usepackage[table]{xcolor}
\definecolor{lightgray}{gray}{0.9}
\usepackage{arydshln}
\usepackage{diagbox}
\newtheorem{remark}{Remark}
\usepackage{subcaption}

\usepackage{here}

\DeclareMathAlphabet{\mathpzc}{OT1}{pzc}{m}{it}
\title{Quaternion tensor low-rank approximation using
a family of non-convex norms}

\author{ A. Zahir \thanks{The UM6P Vanguard Center, Mohammed VI Polytechnic University, Green 
City, Morocco.}
\and K. Jbilou\footnotemark[1] \thanks{Université du Littoral Cote d'Opale, LMPA, 50 rue F. Buisson, 62228 Calais-Cedex, France.}
\and A. Ratnani\footnotemark[1]}

\begin{document}
\maketitle

\begin{abstract}
In this paper, we propose a new approaches for low rank approximation of quaternion tensors \cite{chen2019low,zhang1997quaternions,hamilton1866elements}. The first method uses quasi-norms to approximate the tensor by a low-rank tensor using the QT-product \cite{miao2023quaternion}, which generalizes the known L-product to N-mode quaternions. The second method involves Non-Convex norms to approximate the Tucker and TT-rank for the completion problem. We demonstrate that the proposed methods can effectively approximate the tensor compared to the convexifying of the rank, such as the nuclear norm. We provide theoretical results and numerical experiments to show the efficiency of the proposed methods in the Inpainting and Denoising applications.
\end{abstract}


\section{Introduction}
Low-rank matrix approximation (LRMA) is a fundamental mathematical tool used in a wide range of real-world applications, including image denoising, inpainting, and deblurring. LRMA techniques generally fall into two broad categories: matrix factorization-based approaches and matrix rank minimization methods.
\noindent
Matrix factorization decomposes a matrix into lower-dimensional factors, improving computational efficiency and reducing noise. A well-known example is Non-negative Matrix Factorization (NMF) \cite{lee1999learning}, which approximates a data matrix using two non-negative factor matrices. Another prominent method is Robust Principal Component Analysis (RPCA) \cite{candes2011robust}, which decomposes a matrix into a low-rank component and a sparse component, enhancing robustness to outliers. Low-Rank Matrix Completion (LRMC) \cite{cai2010singular} has also been widely studied for recovering missing entries from partially observed matrices under a low-rank assumption.\\
This paper focuses on LRMC and RPCA, particularly their limitations in handling high-dimensional and color data.
The LRMC problem is particularly relevant in computer vision and graphics, where it is commonly referred to as the image and video inpainting problem. The objective is to reconstruct missing or occluded content by exploiting the assumption that the underlying data has a low-rank structure. However, most existing recovery models are primarily designed for grayscale images, while color images and videos require more sophisticated handling due to their higher dimensionality and inter-channel correlations.
\noindent
Traditional matrix-based approaches often ignore essential structural information in color data. For instance, color images are typically represented by three separate RGB channels. Conventional methods often convert these to grayscale or flatten them into matrices, which can result in information loss.
\noindent
A more effective approach is to model color images as third-order tensors and videos as fourth-order tensors. This representation preserves spatial and spectral correlations and avoids unnecessary reshaping operations. Similarly, hyperspectral images, which contain multiple spectral bands, are naturally modeled as higher-order tensors.
\noindent
To further enhance low-rank modeling, Tensor Robust Principal Component Analysis (TRPCA) \cite{lu2016tensor, hillar2013most, lu2019tensor} extends RPCA to third-order tensors. TRPCA decomposes a tensor into a low-rank component and a sparse component, supporting applications such as video inpainting and denoising. Unlike matrix-based RPCA, TRPCA employs the tensor nuclear norm (TNN) within the t-SVD framework, enabling more effective tensor rank minimization.
\noindent
An advantage of TRPCA is that it does not require prior knowledge of the rank or sparsity levels. Convex optimization techniques, such as the Alternating Direction Method of Multipliers (ADMM), are typically used to solve the resulting problem efficiently.\\
Recent work has shown that quaternion representations are particularly effective for color images and videos. By encoding RGB channels as quaternion entries, quaternion-based methods preserve color structure and inter-channel dependencies. These models have demonstrated success in various vision tasks, including denoising, histopathological image analysis, and color object detection \cite{shi2019quaternion, gai2015denoising}.
\noindent
Minimizing the rank function is an NP-hard problem. A common approach is to use the nuclear norm as a convex surrogate \cite{candes2012exact}, but it treats all singular values equally and may fail to approximate the true rank accurately \cite{chen2019low}. To overcome this, several non-convex surrogate functions have been introduced, including the Geman function \cite{geman1995nonlinear} and Laplace penalty \cite{fazel2003log}, which have shown improved results in grayscale image recovery.
\noindent
To address the limitations of existing approaches, this paper introduces a series of quaternion-based low-rank tensor models using non-convex surrogate functions. The key contributions of this work are:
\begin{itemize}
\item A novel model for quaternion tensor completion using non-convex Tucker rank minimization, named Low-Rank Quaternion Tensor Completion via Non-Convex Tucker Rank (LRQTC-NCTR).
\item A theoretical analysis of the local convergence properties of the completion method.
\item A second model for quaternion tensor completion based on TT-rank, named Low-Rank Quaternion Tensor Completion via Non-Convex TT-Rank (LRQTC-NCTTR).
\item A third model for denoising using non-convex surrogate functions, named Tensor Robust Principal Component Analysis via Non-Convex Norms (TRPCA-NC).
\item Comprehensive experiments demonstrating the superiority of the proposed methods for color video completion and denoising applications compared to conventional approaches.
\end{itemize}
\noindent
The remainder of this paper is organized as follows. Section \ref{Sec:Preliminaries} introduces some preliminaries of quaternion tensors, Section \ref{Sec:LRTCP} introduces the Completion problem, and Section \ref{Sec:LRTP} introduces the second method for the Denoising problem. We propose an extension to the proposed methods in Section \ref{sec:Extensions} before we evaluate the performance of the method in Section \ref{Sec:Experiments}, and conclude this paper in Section \ref{Sec:Conclusion}.

\section{Preliminaries}\label{Sec:Preliminaries}
This section covers basic quaternion algebra, notations, and multidimensional related theories, including products, norms, and transformations.

\subsection{Basic quaternion algebra and notations}
Quaternions, first introduced by William Rowan Hamilton in 1843, extend the concept of complex numbers to a four-dimensional space. A quaternion $\dot{q} \in \mathbb{H}$ \cite{hamilton1866elements}, (denoted also as $\mathbb{Q}$), consists of a real part and three imaginary units
\begin{equation*}
\dot{q}= q_0+q_1\mathbf{i}+q_2\mathbf{j}+q_3\mathbf{k},
\end{equation*}
where $q_l \in \mathbb{R}, l \in [|0,3|]$ and $\mathbf{i},\mathbf{j},\mathbf{k}$ are the imaginary units, verifying
\begin{equation*}
\begin{aligned}
&\mathbf{i}^2=\mathbf{j}^2=\mathbf{k}^2=-1,\\
&\mathbf{i}\mathbf{j}=-\mathbf{j}\mathbf{i}=\mathbf{k}, \; \mathbf{j}\mathbf{k}=-\mathbf{k}\mathbf{j}=\mathbf{i}, \;
\mathbf{k}\mathbf{i}=-\mathbf{i}\mathbf{k}=\mathbf{j}.
\end{aligned}
\end{equation*}
The quaternion skew-field $\mathbb{H}$ is an associative but non-commutative algebra of rank 4 over $\mathbb{R}$, and 2 over $\mathbb{C}$.
For a quaternion $\dot{q}$, we denote $\mathfrak{R}(\dot{q}):=q_0$, and $\mathfrak{J}(\dot{q}):= q_1\mathbf{i}+q_2\mathbf{j}+q_3\mathbf{k}$ as the real and imaginary part of a quaternion, respectively. If the real part is zero, it is called, a pure quaternion.
The conjugate of $\dot{q}$ is given by $\dot{\widebar{q}}=\mathfrak{R}(\dot{q})-\mathfrak{J}(\dot{q})$, and its norm is $||\dot{q}||=\sqrt{\dot{q}\dot{\widebar{q}}}=\sqrt{q_0^2+q_1^2+q_2^2+q_3^2}$.\\
Let $(.)^T,(.)^H$ denote the transpose and the conjugate transpose, respectively.\\
We use $\mathbf{1}$ to denote a vector of ones.

\subsection{Quaternion matrix and tensor}
We denote $\mathbb{H}^{I_1 \times \ldots \times I_N}$ the collection of $I_1 \times \ldots \times I_N$ multi-array (tensors) with quaternion entries. 
An $N$-th order quaternion tensor $\dot{\mathcal{Q}}$ is represented as $\dot{\mathcal{Q}}=\mathcal{Q}_0+\mathcal{Q}_1 \mathbf{i} +\mathcal{Q}_2 \mathbf{j}+\mathcal{Q}_3 \mathbf{k}$, with $\mathcal{Q}_i \in \mathbb{R}^{I_1 \times \ldots \times I_N}, i \in [|0,3|]$. Another form that is common to use, called the Cayley-Dickson Form, which represents the tensor as sum of two components, i.e., $\dot{\mathcal{Q}}=\mathcal{Q}_a+\mathcal{Q}_b\mathbf{j}$, with $\mathcal{Q}_a=\mathcal{Q}_0+\mathcal{Q}_1\mathbf{i},$ and $\mathcal{Q}_b=\mathcal{Q}_2+\mathcal{Q}_3\mathbf{i}$.

\begin{theorem}[QSVD \cite{zhang1997quaternions}]
Given $\dot{Q} \in \mathbb{H}^{m \times n}$, there exist two unitary matrices $\dot{U} \in \mathbb{H}^{m \times m},\dot{V} \in \mathbb{H}^{n \times n}$ and a real rectangular diagonal matrix $S=\operatorname{diag}(\sigma_i)$, such that
\begin{equation*}
\dot{Q}=\dot{U} S \dot{V}^H.
\end{equation*}
The decomposition is called quaternion singular value decomposition (QSVD).\\
\end{theorem}
The unitary quaternion matrix $\dot{U}$ verifies $\dot{U}^H \dot{U}=\dot{U} \dot{U}^H=I$, with $I$ being the identity matrix.\\
It is shown \cite{xu2015vector} how to compute the QSVD using the isomorphic complex morphism of the Cayley-Dickson form $ \mathbb{C}^{2m \times 2n}$ and SVD of a complex matrix.

\begin{definition}
Given the QSVD $\dot{Q}=\dot{U} S \dot{V}^H$, with $S=\operatorname{diag}(\sigma_i)$. The nuclear norm of $\dot{Q}$ is $||\dot{Q}||_* := \sum_i \sigma_i$, i.e., the sum of its singular values.
\end{definition}

\begin{definition}[Mode-k unfolding]
Given a $N$-th order quaternion tensor $\dot{\mathcal{Q}} \in \mathbb{H}^{I_1 \times \ldots \times I_N} $, the mode-k unfolding (also known as mode-k matricization or flattening) is defined as a quaternion matrix $\dot{Q}_{[k]} \in \mathbb{H}^{I_k \times (I_2 \ldots I_{k-1} I_{k+1} \ldots I_N)}$ with entries
\begin{equation*}
\dot{Q}_{[k]}\left(i_k, i_1 \ldots i_{k-1} i_{k+1} \ldots i_N\right)=\dot{q}_{i_1, i_2, \ldots, i_N},
\end{equation*}
where $\dot{q}_{i_1, i_2, \ldots, i_N}$ is the $\left(i_1, i_2, \ldots, i_N\right)$ th-entry of $\dot{\mathcal{Q}}$.
\end{definition}
Conversely, we can define the inverse of this operation as $\operatorname{fold}_k(\dot{Q}_{[k]}) =\dot{\mathcal{Q}}$.
\\
The tensor element $(i_1, i_2, \ldots, i_{k_1},i_k,i_{k+1},\ldots, i_N)$ is mapped to the matrix element $(i_k,j)$ such that
\begin{equation*}
j= 1 + \sum_{l \neq k}^N (i_l-1) J_l, \quad J_l= \prod_{m \neq k}^{l-1} I_m.
\end{equation*}

\begin{definition}[Tucker Rank]\cite{kolda2009tensor}
Given a quaternion tensor $\dot{\mathcal{Q}} \in \mathbb{H}^{I_1 \times I_2 \times \ldots \times I_N}$, the Tucker rank is defined as
\begin{equation}
\operatorname{rank}_{\text {tucker}}(\dot{\mathcal{Q}})=\left(\operatorname{rank}\left(\dot{Q}_{[1]}\right), \operatorname{rank}\left(\dot{Q}_{[2]}\right), \ldots, \operatorname{rank}\left(\dot{Q}_{[N]}\right)\right).
\end{equation}
\end{definition}
In quaternion domain, one should pay attention to the usual products, as due to its non commutativity, we have commonly, more than one definition for the products. We define the Right inner product, or simply, the inner product as follows,
\begin{definition}
The inner product of two $N$-th order quaternions tensors $\dot{\mathcal{Q}},\dot{\mathcal{P}}$ of same size, is defined as
\begin{equation*}
\langle \dot{\mathcal{Q}}, \dot{\mathcal{P}}\rangle := \sum_{i_1,\ldots,i_N} \dot{\widebar{\mathcal{Q}}}_{i_1 i_2 \ldots i_N} \dot{\mathcal{P}}_{i_1 i_2 \ldots i_N},
\end{equation*}
The corresponding Frobenius norm 
$||\dot{\mathcal{Q}}||_F := \sqrt{\langle \dot{\mathcal{Q}}, \dot{\mathcal{Q}}\rangle} =\sqrt{\sum_{i_1,\ldots,i_N} ||\dot{q}_{i_1 i_2 \ldots i_N}||^2}$.
\end{definition}
We also define the frontal slices of a quaternion tensor $\dot{\mathcal{Q}} \in \mathbb{H}^{I_1 \times I_2 \times \ldots \times I_N}$, denoted as $\dot{\mathcal{Q}}(:,:,i_3,\ldots,i_N)$. For convenience, this is also represented as $\dot{\mathcal{Q}}^{(i_3)}$ when referring to third-order tensors specifically. This notation helps in simplifying the representation of tensor slices, which will be utilized in the product definitions discussed later.
\\
Additionally, we denote $\sigma_i(.)$ to represent the i-th singular value of a matrix or tensor. This notation will be important for discussing tensor decomposition and related operations in the subsequent sections.

\subsection{The QT-product}
\begin{definition}\cite{miao2023quaternion}
Given two $N$-th order ($N \geq 3$) quaternion tensors $\dot{\mathcal{Q}} \in \mathbb{H}^{I_1 \times J \times I_3 \ldots I_N},\dot{\mathcal{P}} \in \mathbb{H}^{J \times I_2 \times \ldots I_N}$ and invertible matrices $\dot{M}_i \in \mathbb{H}^{I_i \times I_i}, i \in [|3,N|]$, the $\times_{QT}$ product is defined as 
\begin{equation*}
\dot{\mathcal{Q}} \times_{QT} \dot{\mathcal{P}}=\left(\widehat{\dot{\mathcal{Q}}} \times_{QF} \widehat{\dot{\mathcal{P}}} \right) \times_3 \dot{M}_3^{-1} \times_4 \dot{M}_4^{-1} \dots \times_N \dot{M}_N^{-1},
\end{equation*}
with $\times_{QT}$ is the face-wised product defined \cite{kernfeld2015tensor} by frontal slices as
\begin{equation*}
\widehat{\dot{\mathcal{Q}}} \times_{QF} \widehat{\dot{\mathcal{P}}}(:,:,i_3,\ldots,i_N)= \widehat{\dot{\mathcal{Q}}}(:,:,i_3,\ldots,i_N)\widehat{\dot{\mathcal{P}}}(:,:,i_3,\ldots,i_N).
\end{equation*}
and $\widehat{(.)}$ is the $(\dot{M}_i)_{i=3,\ldots,N}$ transformation, i.e.,
\begin{equation*}
\widehat{\dot{\mathcal{Q}}}=\dot{\mathcal{Q}} \times_3 \dot{M}_3 \times_4 \dot{M}_4 \dots \times_N \dot{M}_N.
\end{equation*}
\end{definition}
Next, we define the associated concepts, as conjugate transpose, the identity quaternion tensor, and the unitary quaternion tensor under the above defined QT-product.
\begin{definition}
Let $\dot{\mathcal{Q}} \in \mathbb{H}^{I_1 \times \ldots \times I_N}$ and $\dot{\mathcal{U}} \in \mathbb{H}^{J \times J \times I_3 \times \ldots \times I_N}$, then, 
\begin{itemize}
\item The conjugate transpose of $\dot{\mathcal{Q}}$ satisfies
$\widehat{\dot{\mathcal{Q}}}^H(:,:,i_3,\ldots,i_N)=\left(\widehat{\dot{\mathcal{Q}}}(:,:,i_3,\ldots,i_N) \right)^H$.
\item The identity quaternion tensor $\dot{\mathcal{I}} \in \mathbb{H}^{J \times J \times I_3 \ldots \times I_N}$ satisfies $\widehat{\dot{\mathcal{I}}}(:,:,i_3,\ldots,i_N)=I_J$.
\item $\dot{\mathcal{U}}$ is unitary if it satisfies $\dot{\mathcal{U}} \times_{QT} \dot{\mathcal{U}}^H =\dot{\mathcal{I}}$.
\item $\dot{\mathcal{Q}}$ is f-diagonal if its frontal slices in the transformed domain are diagonal.
\end{itemize}
\end{definition}
Next, we define the rank and norms related to the defined product.
\begin{theorem}[QT-SVD]\cite{miao2023quaternion}
Given $\dot{\mathcal{Q}} \in \mathbb{H}^{I_1 \times \ldots \times I_N}$, there exist two unitary tensors $\dot{\mathcal{U}},\dot{\mathcal{V}}$ and a f-diagonal tensor $\mathcal{S}$ of same size as $\dot{\mathcal{Q}}$ such that
\begin{equation}
\dot{\mathcal{Q}}=\dot{\mathcal{U}} \times_{QT} \mathcal{S} \times_{QT} \dot{\mathcal{V}}^H. 
\end{equation}
\end{theorem}

\begin{definition}
Given $\dot{\mathcal{Q}}$ and its QT-SVD $\dot{\mathcal{Q}}=\dot{\mathcal{U}} \times_{QT} \mathcal{S} \times_{QT} \dot{\mathcal{V}}^H$. We define,
\begin{itemize}
\item The Qt-rank \cite{miao2023quaternion} of $\dot{\mathcal{Q}}$ is the number of non zero tubes $\mathcal{S}(k,k,:,\ldots,:)$.
\item For third order tensors, the quaternion Tensor Nuclear Norm (QTNN) \cite{yang2024quaternion} as
$||\dot{\mathcal{Q}}||_*=\sum_{i,j} \sigma_j(\widehat{\dot{\mathcal{Q}}}^{(i)})=\sum_{i,j} \sigma_j(\mathcal{S}^{(i)}).$
\end{itemize}
\end{definition}
\noindent
Other norms as Quaternion Tensor Truncated Nuclear Norm (QT-RNN), the quaternion tensor Logarithmic norm (QTLN) can be defined similarly \cite{yang2024quaternion}.
The following algorithm computes the QT-SVD.
\begin{algorithm}[H]
\caption{QT-SVD of third order tensor}
\hspace*{\algorithmicindent} \textbf{Input:} $\dot{\mathcal{T}} \in \mathbb{H}^{I_1 \times I_2 \times I_3}$, $\dot{M} \in \mathbb{H}^{I_3 \times I_3}$.\\
\hspace*{\algorithmicindent} \textbf{Output:} $\dot{\mathcal{U}} \in \mathbb{H}^{I_1 \times I_1 \times I_3}, \mathcal{S} \in \mathbb{R}^{I_1 \times I_2 \times I_3}, \dot{\mathcal{V}} \in \mathbb{H}^{I_2 \times I_2 \times I_3}$.
\begin{algorithmic}[1]
\State $\widehat{\dot{\mathcal{Q}}}:=L(\dot{\mathcal{Q}})=\dot{\mathcal{Q}} \times_3 \dot{M}$.\Comment{Transformed domain}
\For{$i= 1,2,\ldots,I_3$}
\State $\left[\widehat{\dot{\mathcal{U}}}^{(i)},\widehat{\mathcal{S}}^{(i)},\widehat{\dot{\mathcal{V}}}^{(i)}\right]= \text{Q-SVD}(\widehat{\dot{\mathcal{Q}}}^{(i)})$. 
\EndFor
\State $\dot{\mathcal{U}}=L^{-1}(\widehat{\dot{\mathcal{U}}}), \mathcal{S}=L^{-1}(\widehat{\mathcal{S}}), \dot{\mathcal{V}}=L^{-1}(\widehat{\dot{\mathcal{V}}})$.
\end{algorithmic}
\end{algorithm}
It has been shown \cite{pan2024block}, a quaternion circulant matrix can be diagonalised, and can not be diagonalised. In the paper, this result can be applied to compute a fast QT-SVD of a quaternion tensor, when the transformation is the quaternion discrete Fourier transform $\dot{F}_{\dot{u}}$, thus, the number of SVD computed in the algorithm above is shortened by approximately the half.

\subsection{Introducing the Non-Convex surrogate functions}

The application of non-convex penalty functions has been explored to improve the recovery of sparse vectors, particularly through approximations of the $l_0$ norm. Notable examples of such penalty functions include the Smoothly Clipped Absolute Deviation (SCAD) \cite{fan2001variable}, the Logarithmic function \cite{fazel2003log}, and the Geman function \cite{geman1995nonlinear}. Many of these approaches have been adapted to approximate the rank function, leading to the development of methods such as the Weighted Nuclear Norm \cite{gu2014weighted}, the Schatten $p$-norm \cite{nie2012low}, and the Weighted Schatten $p$-norm, which combines the properties of the previous two \cite{xie2016weighted}. These non-convex functions have demonstrated superior performance compared to the Standard Nuclear Norm in various numerical experiments.
\\
Building on the work of \cite{chen2019low}, which generalized several existing quasi-norm functions as discussed in \cite{chen2017Denoising, kang2015robust, xie2016weighted}, we extend this investigation to the quaternion tensor case. Inspired by the tensor extension of these functions to the complex tensor domain via the T-product, as introduced in \cite{cai2019tensor}, we explore their application within the quaternion tensor framework. This exploration leverages the QT-product, a generalization of the tensor product family, which is applicable to $N$-th order tensors. Our goal is to evaluate the effectiveness of these non-convex penalty functions in the context of quaternion tensors, aiming to enhance tensor recovery and Denoising capabilities.

\begin{table}[H]
\centering
\begin{tabular}{l| c| c }
\hline
Name & $\phi_{\gamma}(x)$, $x \geq 0$, $\gamma>0$ & $\partial \phi_{\gamma}(x)$\\ \hline
Geman \cite{geman1995nonlinear} & $\frac{(1+\gamma) x}{x+\gamma}$ & $\frac{(1+\gamma) \gamma}{(x+\gamma)^2}$ \\ \hline
Laplace \cite{trzasko2009highly} & $1-e^{-\frac{x}{\gamma}} $ & $\frac{1}{\gamma}e^{-\frac{x}{\gamma}}$ \\ \hline
Logarithm \cite{fazel2003log} & $\operatorname{log}( \gamma +x)$ & $\frac{1}{\gamma +x}$ \\ 
\hline
Weighted Nuclear norm \cite{gu2014weighted} & $w_i x$ & $w_i$ \\ 
\hline
Schatten p-norm \footnotemark[1] \cite{nie2012low} & $x^p$ & $px^{p-1}$ \\ 
\hline
Weighted Schatten p-norm \footnotemark[1] \cite{xie2016weighted} & $w_i x^p$ & $w_ip x^{p-1}$ \\ 
\hline
\end{tabular}
\caption{\small{Popular Non-Convex surrogate functions of rank and their subgradients. All these functions share the common properties: Concave, and monotonically increasing in $\mathbb{R}^+$}.}
\label{tab_nonpenlty}
\end{table}
\footnotetext[1]{It is convex for $p > 1$ and concave for $p < 1$. We are interested in the latter case.}
For simplicity, we denote by $\phi_{gamma}$ any function of the above table, and we define the two extended quasi-norms, the first extends the Tucker rank, and the second extends the T-QT rank. 
\begin{definition}
\begin{itemize}
\item Let $\dot{Q} \in \mathbb{H}^{I_1 \times I_2}$, we define the $\phi_{\gamma}$ norm as
$||\dot{Q}||_{\phi_{\gamma}}:= \sum_i \phi_\gamma (\sigma_i(\dot{Q}))$.
\item Let $\mathcal{\dot{Q}} \in \mathbb{H}^{I_1 \times I_2 \times I_3}$, the QT-$\phi_{\gamma}$ norm is defined as, $||\mathcal{\dot{Q}}||_{\phi_{\gamma}}=\sum_i ||\widehat{\dot{\mathcal{Q}}}^{(i)}||_{\phi_{\gamma}} =\sum_{i,j} \phi_{\gamma}(\sigma_j^{(i)})$.
\end{itemize}
\end{definition}
The Non-Convex surrogate function norms has some important properties, such as, the Unitary invariant, i,e $||\dot{Q}||_{\phi_{\gamma}}=||S||_{\phi_{\gamma}}$, where $S$ is the rectangular diagonal matrix of the QSVD decomposition, as well as convergence to the rank or to the nuclear norm for specific parameters for some of these methods. The unitary invariance is also established for the QT-$\phi_{gamma}$.
\\
Next, we give the propositions that are needed to solve the optimization problem.
\begin{proposition}\cite{chen2019low}
Given $\dot{Q}=\dot{U}diag(\sigma_i) \dot{V}^H$, the solution of the LRQA is the following
\begin{equation}\label{eq:LRQA_nc}
 \arg \min_{\dot{X}} \dfrac{1}{2}||\dot{X}-\dot{Q}||_F^2 + \lambda ||\dot{X}||_{\phi_\gamma}=\dot{U}diag(\sigma_i^*) \dot{V}^H.
\end{equation}
where $\sigma_i^* \in \mathbb{R}$, is the solution of the following problem, called the Moreau Yosida operator,
\begin{equation}\label{2}
\sigma_i^*=\arg \min_{\dot{x}} \frac{1}{2} ||\dot{x}- \sigma_i||_2^2 + \lambda \phi_\gamma(\dot{x}).
\end{equation}
The solution of the problem \ref{eq:LRQA_nc} will be denoted by $\mathfrak{S}_{\lambda,\phi_\gamma}(\dot{Q})$.
\end{proposition}
\\
A common method to solve the problem (minimize $f(x)+g(x)$) with Non-Convex regularizer, is Difference of Convex functions (DC) \cite{gasso2009recovering}, it minimizes the Non-Convex function $-g(.)$, based on the assumption that both the functions $(f(.),-g(.)$ are convex. In each iteration, DC programming linearizes $-g(.)$ at $x=x^t$, and minimizes a relaxed function.

\begin{proposition}\label{3}
The solution of Equation \ref{2} using DC method, gives the iteration solution
\begin{equation}
\begin{aligned}
\sigma^{(t+1)}=& \arg \min_{x \geq 0} \frac{1}{2} ||x- \sigma^{(t)}||_2^2 + \lambda \partial \phi_\gamma(\sigma^{(t)})x\\
=& \left[ \left(0,\sigma^{(t)}-\lambda \partial \phi_\gamma(\sigma^{(t)})\right) \right]_+
\end{aligned} 
\end{equation}
Where the operator $[.]_+=\max(.,0)$.
\end{proposition}
Next, we will show the propositions that are needed in the Denoising problem. 
\begin{proposition}
Given the QT-SVD $\dot{\mathcal{Q}}=\dot{\mathcal{U}} \times_{QT} \mathcal{S} \times_{QT} \dot{\mathcal{V}}^H \in \mathbb{H}^{I_1 \times I_2 \times I_3}$, with $\mathcal{S}^{(i)}=(diag(\sigma_j^{i}))_{j=1,\ldots,(\min{I_1,I_2})}$, and $\lambda \in \mathbb{R}$, the solution of the RPCA problem is the following
\begin{equation}\label{eq:RPCA_nc}
 \arg \min_{\mathcal{\dot{X}}} \dfrac{1}{2}||\mathcal{\dot{X}}-\mathcal{\dot{Q}}||_F^2 + \lambda ||\mathcal{\dot{X}}||_{\phi_\gamma}=\dot{\mathcal{U}} \times_{QT} \mathcal{S}^* \times_{QT} \dot{\mathcal{V}}^H.
\end{equation}
 where $\mathcal{S}^*$ is f-diagonal tensor such that its frontal slices $\mathcal{S}^{*(i)}=(diag(\sigma_j^{*(i)}))_{j=1,\ldots,(\min{I_1,I_2})}$, and $\sigma_j^{*(i)}$ are the solution of the problem \ref{2}.
\end{proposition}
\begin{proof}
The transformation to the QT domain preserves the norm, thus, the problem becomes
\begin{equation*}
\sum_i ||\widehat{\dot{\mathcal{X}}}^{(i)} -\widehat{\dot{\mathcal{Q}}}^{(i)} ||_F^2 + \lambda ||\widehat{\dot{\mathcal{X}}}^{(i)} ||_{\phi_\gamma},
\end{equation*}
which leads to solving $I_3$ independent problems;
$||\widehat{\dot{\mathcal{X}}}^{(i)} -\widehat{\dot{\mathcal{Q}}}^{(i)} ||_F^2 + \lambda ||\widehat{\dot{\mathcal{X}}}^{(i)} ||_{\phi_\gamma}$, where the solution is given in proposition \ref{3}.
\end{proof}

The solution of the problem \ref{eq:RPCA_nc} will be denoted by $\mathfrak{S}_{\lambda,\phi_\gamma}(\mathcal{\dot{Q}})$. The following algorithm solves the problem.
\begin{algorithm}[H]
\caption{Solve \ref{eq:RPCA_nc} via DC}
\label{algo:solve_via_DC}
\hspace*{\algorithmicindent} \textbf{Input:} $\dot{\mathcal{Q}} \in \mathbb{H}^{I_1 \times I_2 \times I_3},$ $\phi_\gamma$ (Surrogate function), $\lambda$ (threshold).\\
\hspace*{\algorithmicindent} \textbf{Output:} $\mathcal{X}$.
\begin{algorithmic}[1]
\State $\left[ \dot{\mathcal{U}}, \mathcal{S}, \dot{\mathcal{V}} \right]$= QT-SVD$(\dot{\mathcal{Q}})$.
\For{$i=1,2,\ldots,I_3$} \Comment{Can be parallelised.}
\For{$j=1,2,\ldots,\min{I_1,I_2}$}\Comment{Can be parallelised.}
\While{not converged} \Comment{Stopping criterion}
\State Compute $\sigma_j^{*(i)}$ from $\sigma_j^{(i)}$ via \ref{3}.
\EndWhile
\EndFor
\State $\mathcal{S}^{*(i)}=diag(\sigma_j^{*(i)})$.
\EndFor
\State $\mathcal{X}=\dot{\mathcal{U}} \times_{QT} \mathcal{S}^* \times_{QT} \dot{\mathcal{V}}^H$.
\end{algorithmic}
\end{algorithm}

\begin{proposition}[\cite{yang2024quaternion}]
Given a third order quaternion tensor $\dot{\mathcal{Q}}$ and $\lambda \in \mathbb{R}$, we have, 
\begin{equation*}
\begin{aligned}
\mathfrak{Shrink}_{\lambda}(\dot{\mathcal{Q}}) 
&:= \arg \min_{\dot{\mathcal{X}}} ||\dot{\mathcal{X}}-\dot{\mathcal{Q}}||_F^2 + \lambda ||\dot{\mathcal{X}}||_1 \\
&= \operatorname{sign} (\dot{\mathcal{Q}}) \max( |\dot{\mathcal{Q}}| -2\lambda,0),
\end{aligned}
\end{equation*} 
with the $|.|$ is the element wise norm, and the function $\operatorname{sign}$ is an element wise operator defined as,
\begin{equation*}
\operatorname{sign}(\dot{q}_{i,j,k})=
\begin{cases}
\dfrac{\dot{q}_{i,j,k}}{||\dot{q}_{i,j,k}||} \text{ if } |\dot{q}_{i,j,k}| \neq 0,\\
0, \text{ otherwise.} 
\end{cases}
\end{equation*}
\end{proposition}

\section{Low rank tensor completion}\label{Sec:LRTCP}
Given $\dot{\mathcal{O}}$ the observed quaternion tensor, $\Omega$ the index set for the observed elements.
Using Tucker rank, \cite{miao2020low} solves the low rank problem by the Tucker rank which can be formulated \cite{liu2012tensor} as
\begin{equation}\label{Prob_tucker_rank}
\begin{aligned} 
& \min_{\dot{\mathcal{Q}}} \sum_k^N \alpha_k rank(\dot{Q}_{[k]}) \\
& \text{s.t } 
\dot{\mathcal{Q}}_\Omega=\dot{\mathcal{O}}_\Omega,
\end{aligned}
\end{equation}
with the non-negatives weights $\alpha_k$, satisfying $\sum_k^N \alpha_k =1$, and the linear operation $._\Omega$ keeps the entries in $\Omega$ and zeros out others.\\
The problem \ref{Prob_tucker_rank} is NP-hard, thus, inspired by the convex relaxation, \cite{liu2012tensor} solves the following problem,
\begin{equation}
\begin{aligned} 
& \min_{\dot{\mathcal{Q}}} \sum_k^N \alpha_k ||\dot{Q}_{[k]}||_* \\
& \text{s.t }
\dot{\mathcal{Q}}_\Omega=\dot{\mathcal{O}}_\Omega,
\end{aligned}
\end{equation}
which is solved for the quaternion entries in \cite{miao2020low}. \\
However, although the problem become convex and easy to solve, It might not approximate well the rank, and would give sub-optimal results, thus, we propose the following model that uses Non-Convex surrogate function, named, Low-rank quaternion tensor completion via Non-Convex Tucker rank (\textbf{LRQTC-NCTR}).
\begin{equation}
\begin{aligned} 
& \min_{\dot{\mathcal{Q}}} \sum_k^N \alpha_k \phi_\gamma(\dot{Q}_{[k]}) \\
& \text{s.t } \dot{\mathcal{Q}}_\Omega=\dot{\mathcal{O}}_\Omega.
\end{aligned}
\end{equation}

\subsection{Solution of the proposed model}
We solve the problem using Alternating Direction Method of Multipliers (ADMM) framework with Variable splitting. To make the problem separable, we add the auxiliary variables $\dot{Q}_{k}$ of appropriate size as follows,
\begin{equation}\label{Prob_tucker_nc}
\begin{aligned} 
& \min_{\dot{\mathcal{P}},\dot{Q}_{k}} \sum_k^N \alpha_k \phi_\gamma(\dot{Q}_{k}) \\
& \text{s.t }
\dot{P}_{[k]}=\dot{Q}_{k}, k=1,\ldots, N, \;
\dot{\mathcal{P}}_\Omega=\dot{\mathcal{O}}_\Omega.
\end{aligned}
\end{equation} 
The associated Lagrangian is the following,
\begin{equation*}
\small
\begin{aligned}
\mathfrak{L} \left(\dot{\mathcal{P}},\left\{\dot{Q}_{k}\right\}_{k=1}^N,\left\{\dot{F}_{k}\right\}_{k=1}^N,\left\{\beta_k\right\}_{k=1}^N\right)
=\sum_{k} \left(\alpha_k \phi_\gamma (\dot{Q}_{k})+ \mathfrak{R}\left( \left\langle\dot{F}_{k}, \dot{P}_{[k]}-\dot{Q}_{k}\right\rangle\right) +\frac{\beta_k}{2}\left\|\dot{P}_{[k]}-\dot{Q}_{k}\right\|_F^2\right),
\end{aligned}
\end{equation*}
where $ \left\{ \beta_k\right\}_{k=1}^N$ are the penalty parameters, and 
$\left\{\dot{F}_{k}\right\}_{k=1}^N$ are the Lagrangian multipliers. The iteration scheme proposed to solve is listed next.

\begin{align}
\dot{\mathcal{P}}^{(\tau+1)} &= \arg \min_{\dot{\mathcal{P}}} \mathfrak{L} \left(\dot{\mathcal{P}}^{(\tau)},\left\{\dot{Q}_{k}^{(\tau)}\right\}_{k=1}^N,\left\{\dot{F}_{k}^{(\tau)}\right\}_{k=1}^N,\left\{\beta_k^{(\tau)}\right\}_{k=1}^N\right).\\
\left\{\dot{Q}_{k}^{(\tau+1)}\right\}_{k=1}^N &= \arg \min_{\dot{Q}_{k}} \mathfrak{L} \left(\dot{\mathcal{P}}^{(\tau+1)},\left\{\dot{Q}_{k}^{(\tau)}\right\}_{k=1}^N,\left\{\dot{F}_{k}^{(\tau)}\right\}_{k=1}^N,\left\{\beta_k^{(\tau)}\right\}_{k=1}^N\right).\\
\dot{F}_{k}^{(\tau+1)} &= \dot{F}_{k}^{(\tau)}-\beta_k^{(\tau)} \left(\dot{Q}_{k}^{(\tau+1)}-\dot{P}_{[k]}^{(\tau+1)}\right), k=1,\ldots,N. \label{Lagrangian_Update_F}\\
\boldsymbol{\beta}^{(\tau+1)} &= \min \left(\beta^{\max }, \rho \boldsymbol{\beta}^{(\tau)}\right). \label{Lagrangian_Update_B}
\end{align}
where $\boldsymbol{\beta}^{(\tau)}=\{\beta_1^{(\tau)},\ldots,\beta_N^{(\tau)} \} \in \mathbb{R}^N$.

\textbf{Solving the $\dot{\mathcal{P}}$ sub-problem:} The solution is straightforward, that is
\begin{equation}\label{Inpainting:Update_P_1}
\begin{aligned}
\dot{\mathcal{P}}^{(\tau+1)}& = \arg \min_{\dot{\mathcal{P}}} \mathcal{L}\left(\dot{\mathcal{P}}, \left\{ \dot{\mathcal{Q}}_k^{(\tau)} \right\}^N, \left\{ \dot{F}_{k}^{(\tau)}\right\}^N,\left\{\beta_{[k]}^{(\tau)} \right\}^N \right) \\
& = \arg \min_{\dot{\mathcal{P}}} \sum_{k} \left( \mathfrak{R} \left(\left\langle\dot{F}_{k}^{(\tau)}, \dot{P}_{[k]} -\dot{Q}_{k}^{(\tau)} \right\rangle \right) + \frac{\beta_k^{(\tau)}}{2}\left\|\dot{P}_{[k]}-\dot{Q}_{k}^{(\tau)}\right\|_F^2\right)\\
& = \arg \min_{\dot{\mathcal{P}}} \sum_{k} \frac{\beta_k^{(\tau)}}{2}\left\|\dot{P}_{[k]}-\dot{Q}_{k}^{(\tau)} + \frac{1}{\beta_k^{(\tau)}} \dot{F}_{k}^{(\tau)} \right\|_F^2\\
& = \frac{1}{\mathbf{1}^T \boldsymbol{\beta}^{(\tau)}} \sum_k \operatorname{fold}_k \left(\dot{Q}_{k}^{(\tau)}-\frac{1}{\beta_k^{(\tau)}}\dot{F}_{k}^{(\tau)}\right).
\end{aligned} 
\end{equation}
\\
Next, the observed elements remain unchanged via
\begin{equation}\label{Inpainting:Update_P_2}
\dot{\mathcal{P}}^{(\tau+1)}=\dot{\mathcal{P}}^{(\tau+1)}_{\Omega^C}+\dot{\mathcal{O}}_\Omega.
\end{equation}

\textbf{Solving the $\dot{Q}_{k}$ sub-problems:}
$\dot{Q}_{k}^{(\tau+1)}$ can be solved independently for each $k=1,\ldots,N$, we have
\begin{equation}\label{Inpainting:Update_Q}
\begin{aligned}
\dot{Q}_{k}^{(\tau+1)}& = \arg \min_{\dot{Q}_{k}} \mathcal{L}\left(\dot{\mathcal{P}}^{(\tau+1)},\dot{\mathcal{Q}}_{k}, \left\{ \dot{F}_{k'}^{(\tau)}\right\}^N, \left\{\beta_{[k']}^{(\tau)} \right\}^N \right) \\
& = \arg \min_{\dot{Q}_{k}} \alpha_k ||\dot{Q}_{k}||_{\phi_\gamma}+ \mathfrak{R} \left( \left\langle\dot{F}_{k}^{(\tau)}, \dot{P}_{[k]}^{(\tau+1)}-\dot{Q}_{k}\right\rangle \right) + \frac{\beta_k^{(\tau)}}{2}\left\|\dot{P}_{[k]}^{(\tau+1)}-\dot{Q}_{k}\right\|_F^2 \\
& = \arg \min_{\dot{Q}_{k}} \alpha_k ||\dot{Q}_{k}||_{\phi_\gamma} + \frac{\beta_k^{(\tau)}}{2}\left\|\dot{P}_{[k]}^{(\tau+1)}-\dot{Q}_{k} + \frac{1}{\beta_k^{(\tau)}} \dot{F}_{k}^{(\tau)} \right\|_F^2 \\
&= \mathfrak{S}_{\frac{\alpha_k}{\beta_k^{(\tau)}},\phi_\gamma} \left(\dot{P}_{[k]}^{(\tau+1)}+\frac{1}{\beta_k^{(\tau)}} \dot{F}_k^{(\tau)}\right),
\end{aligned} 
\end{equation}
where the last equality is using the property \ref{eq:LRQA_nc}.
The following algorithm solves the proposed problem \ref{Prob_tucker_nc}.
\begin{algorithm}[H]
\caption{Inpainting using Tucker rank approximation via Non-Convex norms}
\label{algo:solve_Tucker_Denoising}
\hspace*{\algorithmicindent} \textbf{Input:} $\dot{\mathcal{O}}$ (Observed Data), $\Omega$ (Observed index set), $\phi_\gamma$ (Surrogate function norm), $\rho > 1$, and $\beta^{max}$.\\
\hspace*{\algorithmicindent} \textbf{Output:} $\dot{\mathcal{P}}$ (Low rank tensor).
\begin{algorithmic}[1]
\State Initialize $\left\{\dot{Q}_{k}^{(O)}\right\}_{k=1}^N,\left\{\dot{F}_{k}^{(O)}\right\}_{k=1}^N,\left\{\beta_k^{(O)}\right\}_{k=1}^N, \tau=0$.
\While{not converged} \Comment{Stopping criterion.}
\State Update $\dot{\mathcal{P}}^{(\tau+1)}$ via \eqref{Inpainting:Update_P_1} then \eqref{Inpainting:Update_P_2}.
\State Update $\dot{\mathcal{Q}}_{k}^{(\tau+1)},k=1,\ldots,N$ via \eqref{Inpainting:Update_Q}. \Comment{Can be parallelised.}
\State Update $\dot{F}_{k}^{(\tau+1)},k=1,\ldots,N$ via \eqref{Lagrangian_Update_F}. \Comment{Can be parallelised.}
\State Update $\boldsymbol{\beta}^{(\tau+1)}$ via \eqref{Lagrangian_Update_B}.
\State $\tau=\tau+1$.
\EndWhile
\end{algorithmic}
\end{algorithm}

\subsection{Convergence analysis}
In this part, we show the convergence is guaranteed using theorem \ref{thm:cv}, first we need the following two lemmas.
\begin{lemma}
For all $k=1,\ldots,N,$ the sequence $\left\{\dot{F}_{k}^{(\tau)} \right\}_\tau^\infty$ is bounded.
\end{lemma}
\begin{proof}
For a fixed $k$, $\dot{Q}_k$ satisfies first order necessary local optimal condition,
\begin{equation} \label{eq:subgradient_Q_k}
\begin{aligned}
0 \in \; & \partial_{\dot{Q}_k} \mathfrak{L}\left(\dot{\mathcal{P}}^{(\tau+1)},\left\{\dot{Q}_{k}\right\}_{k=1}^N,\left\{\dot{F}_{k}^{(\tau)}\right\}_{k=1}^N,\left\{\beta_k^{(\tau)}\right\}_{k=1}^N\right)|_{\dot{Q}_k^{(\tau+1)}}\\
=& \alpha_k \partial_{\dot{Q}_k} \left( ||\dot{Q}_{k}||_{\phi_\gamma} + \frac{\beta_k^{(\tau)}}{2} \left\|\dot{P}_{[k]}^{(\tau+1)}-\dot{Q}_{k} + \frac{1}{\beta_k^{(\tau)}} \dot{F}_{k}^{(\tau)} \right\|_F^2 \right)|_{\dot{Q}_k^{(\tau+1)}}\\
=& \alpha_k \partial_{\dot{Q}_k}||\dot{Q}_{k}||_{\phi_\gamma}|_{\dot{Q}_k^{(\tau+1)}} + \beta_k^{(\tau)} \left(\dot{Q}_{k}^{(\tau+1)}-\dot{P}_{[k]}^{(\tau+1)}- \frac{1}{\beta_k^{(\tau)}} \dot{F}_{k}^{(\tau)} \right) \\
=& \alpha_k \partial_{\dot{Q}_k} ||\dot{Q}_{k}||_{\phi_\gamma}|_{\dot{Q}_k^{(\tau+1)}} - \dot{F}_{k}^{(\tau+1)}.
\end{aligned}
\end{equation}
As the subgradient $\partial_{\dot{Q}_k}||\dot{Q}_{k}||_{\phi_\gamma}|_{\dot{Q}_k^{(\tau+1)}}$ is bounded, hence, the sequence $\left\{\dot{F}_{k}^{(\tau)} \right\}_\tau^\infty, \; k=1,\ldots,N,$ is bounded.
\end{proof}

\begin{lemma}\label{lemma:condition}
Under the condition $\sum_{\tau}^{\infty} \dfrac{\beta_k^{(\tau)}+\beta_k^{(\tau-1)}}{2 \beta_k^{(\tau-1)^2}}, \; \forall k=1,\ldots,N,$ is bounded (denoted as $\beta$-Convergence condition), then the sequences $\left\{\dot{\mathcal{P}}^{(\tau)} \right\}_\tau^\infty,\left\{ \left\{\dot{Q}_{k}^{(\tau)}\right\}_{k=1}^N \right\}_\tau^\infty$, are bounded.
\end{lemma}
\begin{proof}
We have,
\begin{equation*}
\small
\begin{aligned}
\mathfrak{L} & \left(\dot{\mathcal{P}}^{(\tau)},\left\{\dot{Q}_{k}^{(\tau)}\right\}_{k=1}^N,\left\{\dot{F}_{k}^{(\tau)}\right\}_{k=1}^N,\left\{\beta_k^{(\tau)}\right\}_{k=1}^N\right) 
=
\mathfrak{L}\left(\dot{\mathcal{P}}^{(\tau)},\left\{\dot{Q}_{k}^{(\tau)}\right\}_{k=1}^N,\left\{\dot{F}_{k}^{(\tau-1)}\right\}_{k=1}^N,\left\{\beta_k^{(\tau-1)}\right\}_{k=1}^N\right)\\
&\; \qquad 
+ \sum_k \frac{\beta_k^{(\tau)}-\beta_k^{(\tau-1)}}{2}\left\|\dot{P}_{[k]}^{(\tau)}-\dot{Q}_{k}^{(\tau)}\right\|_F^2 + \mathfrak{R}\left( \left\langle\dot{F}_{k}^{(\tau)}-\dot{F}_{k}^{(\tau-1)}, \dot{P}_{[k]}^{(\tau)}-\dot{Q}_{k}^{(\tau)} \right\rangle\right) \\ 
&\; \qquad
= \mathfrak{L}\left(\dot{\mathcal{P}}^{(\tau)},\left\{\dot{Q}_{k}^{(\tau)}\right\}_{k=1}^N,\left\{\dot{F}_{k}^{(\tau-1)}\right\}_{k=1}^N,\left\{\beta_k^{(\tau-1)}\right\}_{k=1}^N\right) + \sum_k^N \frac{\beta_k^{(\tau)}+\beta_k^{(\tau-1)}}{2 \beta_k^{(\tau-1)^2}} \left\|\dot{F}_{k}^{(\tau)}-\dot{F}_{k}^{(\tau-1)}\right\|_F^2.
\end{aligned}
\end{equation*}
Using the fact that $\dot{\mathcal{P}}^{(\tau+1)}$, and $\left\{\dot{Q}_{k}^{(\tau+1)}\right\}_{k=1}^N$, are the solution of the $\dot{\mathcal{P}}$ sub-problem, and $\left\{\dot{Q}_{k}\right\}_{k=1}^N$ sub-problem, respectively, then,
\begin{equation*}
\small
\begin{aligned}
\mathfrak{L} & \left(\dot{\mathcal{P}}^{(\tau+1)},\left\{\dot{Q}_{k}^{(\tau+1)}\right\}_{k=1}^N,\left\{\dot{F}_{k}^{(\tau)}\right\}_{k=1}^N,\left\{\beta_k^{(\tau)}\right\}_{k=1}^N\right)
\leq 
\mathfrak{L} \left(\dot{\mathcal{P}}^{(\tau)},\left\{\dot{Q}_{k}^{(\tau+1)}\right\}_{k=1}^N,\left\{\dot{F}_{k}^{(\tau)}\right\}_{k=1}^N,\left\{\beta_k^{(\tau)}\right\}_{k=1}^N\right)
\\
&\; \qquad \leq 
\mathfrak{L} \left(\dot{\mathcal{P}}^{(\tau)},\left\{\dot{Q}_{k}^{(\tau)}\right\}_{k=1}^N,\left\{\dot{F}_{k}^{(\tau)}\right\}_{k=1}^N,\left\{\beta_k^{(\tau)}\right\}_{k=1}^N\right)
\\
&\; \qquad \leq
\mathfrak{L} \left(\dot{\mathcal{P}}^{(\tau)},\left\{\dot{Q}_{k}^{(\tau)}\right\}_{k=1}^N,\left\{\dot{F}_{k}^{(\tau-1)}\right\}_{k=1}^N,\left\{\beta_k^{(\tau-1)}\right\}_{k=1}^N\right) + \sum_k^N \frac{\beta_k^{(\tau)}+\beta_k^{(\tau-1)}}{2 \beta_k^{(\tau-1)^2}} \left\|\dot{F}_{k}^{(\tau)}-\dot{F}_{k}^{(\tau-1)}\right\|_F^2
\\
&\; \qquad \leq
\mathfrak{L} \left(\dot{\mathcal{P}}^{(1)},\left\{\dot{Q}_{k}^{(1)}\right\}_{k=1}^N,\left\{\dot{F}_{k}^{(0)}\right\}_{k=1}^N,\left\{\beta_k^{(0)}\right\}_{k=1}^N\right) + \sum_k^N \sum_{t=1}^{\tau} \frac{\beta_k^{(t)}+\beta_k^{(t-1)}}{2 \beta_k^{(t-1)^2}} \left\|\dot{F}_{k}^{(t)}-\dot{F}_{k}^{(t-1)}\right\|_F^2,
\end{aligned}
\end{equation*}
where, the last inequality can be deduced by induction.\\
Using Lemma
\ref{lemma:condition}, along with the $\beta$-Convergence condition, then, the right-hand side is upper bounded. Since each term of the Lagrangian is non-negative, hence, we can deduce the result.
\end{proof}
The convergence theorem is given as follows.
\begin{theorem}\label{thm:cv}
Assume the condition is verified, let $\mathcal{N}^{(\tau)}= \left(\dot{\mathcal{P}}^{(\tau)},\left\{\dot{Q}_{k}^{(\tau)}\right\}_{k=1}^N,\left\{\dot{F}_{k}^{(\tau)}\right\}_{k=1}^N\right)$ be a sequence generated by algorithm \ref{algo:solve_Tucker_Denoising}, then any accumulation point $\mathcal{N}^{(*)}$ satisfies the Karush-Kuhn-Tuker (KKT) conditions as follows
\begin{itemize}
\item $\dot{P}_{[k]}^{*}=\dot{Q}_{k}^{*}, k=1,\ldots, N.$
\item $ \dot{\mathcal{P}}_\Omega^{*}=\dot{\mathcal{O}}_\Omega.$
\item $ \dot{F}_{k}^{*} \in \partial \phi_\gamma (\dot{Q}_{k})|_{\dot{Q}_k^{*}}, k=1,\ldots, N.$
\end{itemize}
\end{theorem}
\begin{proof}
Using Bolzano-Weierstrass theorem, there is at least one accumulation point, we assume without loss of generality that $\left\{ \mathcal{N}^{(\tau)} \right\}_\tau^\infty$ converges to $ \mathcal{N}^{*}$.
\\
We have, for $k=1,\ldots,N$,
\begin{equation*}
\dfrac{\dot{F}_{k}^{(\tau+1)}-\dot{F}_{k}^{(\tau)}}{\beta_k^{(\tau)}} =\dot{Q}_{k}^{(\tau+1)}-\dot{P}_{[k]}^{(\tau+1)},
\end{equation*}
with the fact that $\beta_k^{(\tau)}$ is non-decreasing, then
\begin{equation*}
0 = \lim_{\tau \to \infty}
\dfrac{\dot{F}_{k}^{(\tau+1)}-\dot{F}_{k}^{(\tau)}}{\beta_k^{(\tau)}} = \lim_{\tau \to \infty} \dot{Q}_{k}^{(\tau+1)}-\dot{P}_{[k]}^{(\tau+1)}=\dot{Q}_{k}^{*}-\dot{P}_{[k]}^{*},
\end{equation*}
thus, we have
\begin{equation*}
\dot{Q}_{k}^{*}=\dot{P}_{[k]}^{*}, \; \forall k=1,\ldots,N.
\end{equation*}
From \ref{Inpainting:Update_P_2}, we obtain 
\begin{equation*}
\dot{\mathcal{P}}^{(\tau+1)}=\dot{\mathcal{P}}^{(\tau+1)}_{\Omega^C}+\dot{\mathcal{O}}_\Omega \Rightarrow \; 
\dot{\mathcal{P}}^{*}=\dot{\mathcal{P}}^{*}_{\Omega^C}+\dot{\mathcal{O}}_\Omega \Rightarrow \; \dot{\mathcal{P}}^{*}_{\Omega^C}+ \dot{\mathcal{P}}^{*}_{\Omega}=\dot{\mathcal{P}}^{*}_{\Omega^C}+\dot{\mathcal{O}}_\Omega \Rightarrow \; \dot{\mathcal{P}}^{*}_{\Omega}=\dot{\mathcal{O}}_\Omega.
\end{equation*}
From \ref{eq:subgradient_Q_k}, we conclude that, for all $k=1,\ldots,N$
\begin{equation*}
\lim_{\tau \to \infty} \dot{F}_{k}^{(\tau+1)}= \dot{F}_{k}^{*} \in \alpha_k \partial \phi_\gamma (\dot{Q}^{*}_{k}).
\end{equation*}
\end{proof}

\begin{remark}
The condition of the convergence of the series supposition is not verified with \ref{Lagrangian_Update_B} update, yet, numerically, $\beta_max$ big enough to be considered as infinity, while $\beta_0$ is close to zero, we can consider that $\beta_k^{(\tau+1)}= \rho \beta_k^{(\tau)}, \; \forall k=1,\ldots,N$, by this formula, the condition is verified, since we have for $k=1,\ldots,N$,
\begin{equation*}
\begin{aligned}
\sum_\tau \dfrac{\beta_k^{(\tau)}+\beta_k^{(\tau-1)}}{2 \beta_k^{(\tau-1)^2}} &= \sum_\tau \dfrac{\beta_k^{(\tau)}+\beta_k^{(\tau-1)}}{2 \beta_k^{(\tau-1)^2}}\\
&= \dfrac{1}{2} \sum_\tau \dfrac{1}{\beta_k^{(\tau-2)}}+\dfrac{1}{\beta_k^{(\tau-1)}}\\
&= \dfrac{1}{2\beta_0^{(0)}} \sum_\tau \dfrac{1}{\rho^{\tau-2}}+\dfrac{1}{\rho^{\tau-1}},
\end{aligned}
\end{equation*}
which is bounded, as sum of two geometric series, with common factor $\rho > 1$, thus, 
the condition is verified.
\end{remark}

\begin{remark}
Another similar rank problem, called the TT-rank \cite{oseledets2011tensor}, has shown to be used in Tensor Completion \cite{bengua2017efficient}. The TT-rank is vector rank $(r_1,\ldots,r_{N-1})$, where $r_i$ is the rank of folded tensor onto $i$ modes. \cite{bengua2017efficient} claimed that this formulation is well suited to capture the global correlation of a tensor as it provides the mean of few modes, instead of a single mode, with the rest of the tensor. We can propose a new model based also on the change of rank optimization, to the Non-Convex surrogate functions of the rank, and obtain a new-model called the TT Non-Convex norm, named Low-rank quaternion tensor completion via Non-Convex TT-rank \textbf{LRQTC-NCTTR}. This new model enjoys similar steps ans convergence analysis as \textbf{LRQTC-NCTR}. Thus, we avoid redundancy and only show the result of the second series of the proposed methods on the experiments. 
\end{remark}

\section{Low rank tensor problem}\label{Sec:LRTP}
In this part, we are interesting in the TRPCA which aims to recover form a noised Data $\mathcal{\dot{X}}$, a low rank part $\mathcal{\dot{Q}}$, and a sparse part $\mathcal{\dot{S}}$, such that $\mathcal{\dot{X}}=\mathcal{\dot{Q}}+\mathcal{\dot{S}}$. The problem to be solved is the following
\begin{equation}
\begin{aligned} 
& \min_{\dot{\mathcal{Q}},\dot{\mathcal{S}}} ||\dot{\mathcal{Q}}||_* + \lambda ||\dot{\mathcal{S}}||_1 \\
& \text{s.t } 
\mathcal{\dot{X}}=\mathcal{\dot{Q}}+\mathcal{\dot{S}}.
\end{aligned}
\end{equation}
where $\lambda$ is a balancing parameter between the low rank tensor and sparse part.\\
As the first problem, we propose the Tensor Robust Principal Component Analysis via Non-Convex norms (\textbf{TRPCA-NC}) problem, wich is the following, 
\begin{equation}\label{Prob:Denoising}
\begin{aligned} 
& \min_{\dot{\mathcal{Q}},\dot{\mathcal{S}}} ||\dot{\mathcal{Q}}||_{\phi_\gamma} + \lambda ||\dot{\mathcal{S}}||_1 \\
& \text{s.t } 
\mathcal{\dot{X}}=\mathcal{\dot{Q}}+\mathcal{\dot{S}}.
\end{aligned}
\end{equation}

\subsection{Solution of the proposed model}
We solve the problem using ADMM framework of the associated augmented Lagrangian, that is presented next.
\begin{equation*}
\mathfrak{L}\left(\dot{\mathcal{Q}},\dot{\mathcal{S}},\dot{\mathcal{Y}}\right)= ||\dot{\mathcal{Q}}||_{\phi_\gamma} + \lambda ||\dot{\mathcal{S}}||_1 + \mathfrak{R}\left( \left\langle \mathcal{\dot{Y}}, \mathcal{\dot{X}}-\mathcal{\dot{Q}}-\mathcal{\dot{S}} \right\rangle\right) +\frac{\beta}{2}\left\|\mathcal{\dot{X}}-\mathcal{\dot{Q}}-\mathcal{\dot{S}}\right\|_F^2,
\end{equation*}
where $\mathcal{\dot{Y}}$ is the Lagrangian multiplier, and $\beta$ is the penalty parameter. The iteration scheme to solve is the following.

\begin{align}
\dot{\mathcal{Q}}^{(\tau+1)}&= \arg \min_{\dot{\mathcal{Q}}} \mathfrak{L}\left(\dot{\mathcal{Q}},\dot{\mathcal{S}}^{(\tau)},\dot{\mathcal{Y}}^{(\tau)}\right),\\
\dot{\mathcal{S}}^{(\tau+1)}&= \arg \min_{\dot{\mathcal{S}}} \mathfrak{L}\left(\dot{\mathcal{Q}}^{(\tau+1)},\dot{\mathcal{S}},\dot{\mathcal{Y}}^{(\tau)}\right),\\
\dot{\mathcal{Y}}^{(\tau+1)}&=\dot{\mathcal{Y}}^{(\tau)}+\beta ^{(\tau)} \left(\mathcal{\dot{X}}-\mathcal{\dot{Q}}^{(\tau+1)}-\mathcal{\dot{S}}^{(\tau+1)} \right),\label{P2:Lagrangian_Update_Y} \\
\beta^{(\tau+1)} &= \min \left(\beta^{\max }, \rho \beta^{(\tau)}\right). \label{P2:Lagrangian_Update_B}
\end{align}
\\
\textbf{Solving the $\dot{\mathcal{Q}}$ sub-problem:} The problem becomes,
\begin{equation}\label{eq:Noise:Solve_Y}
\begin{aligned}
\dot{\mathcal{Q}}^{(\tau+1)}&= \arg \min_{\dot{\mathcal{Q}}} 
||\dot{\mathcal{Q}}||_{\phi_\gamma} +\frac{\beta^{(\tau)}}{2}\left\|\mathcal{\dot{Q}}+\mathcal{\dot{S}}^{(\tau)}-\mathcal{\dot{X}} -\dfrac{\mathcal{\dot{Y}}^{(\tau)} }{\beta^{(\tau)}} \right\|_F^2\\
&= \mathfrak{S}_{\frac{1}{\beta^{(\tau)}},\phi_\gamma} \left(-\mathcal{\dot{S}}^{(\tau)}+\mathcal{\dot{X}} +\dfrac{\mathcal{\dot{Y}}^{(\tau)} }{\beta^{(\tau)}} \right).
\end{aligned}
\end{equation}
\noindent
\textbf{Solving the $\dot{\mathcal{S}}$ sub-problem:} The problem becomes,
\begin{equation}\label{eq:Noise:Solve_S}
\begin{aligned}
\dot{\mathcal{S}}^{(\tau+1)}&= \arg \min_{\dot{\mathcal{S}}} \lambda ||\dot{\mathcal{S}}||_1 +\frac{\beta^{(\tau)}}{2}\left\|\mathcal{\dot{S}}+ \mathcal{\dot{Q}}^{(\tau+1)}-\mathcal{\dot{X}} -\frac{\mathcal{\dot{Y}}^{(\tau)} }{\beta^{(\tau)}} \right\|_F^2\\
&=\mathfrak{Shrink}_{\frac{2\lambda}{\beta^{(\tau)}}}\left( -\mathcal{\dot{Q}}^{(\tau+1)}+ \mathcal{\dot{X}} + \frac{\mathcal{\dot{Y}}^{(\tau)} }{\beta^{(\tau)}} \right).
\end{aligned}
\end{equation}
\\
The following algorithm solves the problem \ref{Prob:Denoising}.
\begin{algorithm}[H]
\caption{\textbf{TRPCA-NC}}
\label{algo:solve_QTQ_problem}
\hspace*{\algorithmicindent} \textbf{Input:} $\dot{\mathcal{X}}$ (Noised Data), $\phi_\gamma$ (Surrogate function), $\lambda$, $\rho > 1$, and $\beta^{max}$. \\
\hspace*{\algorithmicindent} \textbf{Output:} $\dot{\mathcal{Q}}$ (Low rank tensor), $\dot{\mathcal{S}}$ (Sparse tensor). 
\begin{algorithmic}[1]
\State Initialize $\dot{\mathcal{S}}^{(O)},\dot{\mathcal{Y}}^{(O)}, \beta^{(0)}, \tau=0$.
\While{not converged} \Comment{Stopping criterion.}
\State Update $\dot{\mathcal{Q}}^{(\tau+1)}$ via \eqref{eq:Noise:Solve_Y}.
\State Update $\dot{\mathcal{S}}^{(\tau+1)}$ via \eqref{eq:Noise:Solve_S}.
\State Update $\dot{\mathcal{Y}}^{(\tau+1)}$ via \eqref{P2:Lagrangian_Update_Y}.
\State Update $\beta^{(\tau+1)}$ via \eqref{P2:Lagrangian_Update_B}.
\State $\tau=\tau+1$.
\EndWhile
\end{algorithmic}
\end{algorithm}
\noindent The Convergence analysis of the problem is similar to the first, thus to avoid redundancy, it is not shown.

\section{Extension work}\label{sec:Extensions}
\textbf{Ket augmentation (\textbf{KA}}) scheme has been used in the literature to enhance the algorithms. It tries to represent a low order tensor $\mathcal{T} \in \mathbb{H}^{I_1 \times \ldots \times I_N}$ into a bigger one $\mathcal{K} \in \mathbb{H}^{J_1 \times \ldots \times J_M}$, where $N \leq M$. Claiming that this new representation can offer some new advantages. This technique was first proposed in \cite{latorre2005image} as a way to use an appropriate block structured addressing scheme to convert a gray scale image into the real ket state of a Hilbert space, which is simply a higher-order tensor.
In \cite{bengua2017efficient}, the procedure is explained on how to transform a colored image $\mathcal{T} \in \mathbb{R}^{I_1 \times I_2 \times I_3}$, where $I_1=I_2=2^N$ and $I_3=3$, into a $N+1$-th order tensor $\mathcal{K} \in \mathbb{R}^{4 \times 4 \ldots 4 \times 3}$, which can be transformed to a quaternion tensor with ease. 
This procedure can be extended also to colored Video represented as third order quaternion tensor, to a higher order quaternion tensor.

\section{Experiments}\label{Sec:Experiments}
In this section, we assess the performance of our proposed method and compare it against several state-of-the-art techniques. The implementation of our algorithm, along with the competing methods, is done using MATLAB. We have utilized the source codes provided in the original papers and adhered to the parameter settings specified therein.
All experiments are conducted on a computer equipped with an AMD Ryzen 6-Core Processor running at 3.80 GHz, 8th Generation, and 32 GB of memory, using MATLAB 2024a.
\\
The images and videos used in the study are initially represented in the form of pure quaternions. Specifically, colored images are encoded as $\mathbb{R}^{H \times W \times 3}$ and videos as $\mathbb{R}^{H \times W \times 3 \times T}$, where $H$ and $W$ denote the height and width of the image, $3$ represents the RGB channels, and $T$ is the number of frames. In quaternion representation, an image is encoded as $\dot{Q}_{x,y}=\textcolor{red}{R(x,y)}\mathbf{i}+\textcolor{green}{G(x,y)}\mathbf{j}+\textcolor{blue}{B(x,y)}\mathbf{k} \in \mathbb{H}^{H \times W}$, and a video as $\dot{\mathcal{Q}} \in \mathbb{H}^{H \times W \times T}$.
\\
To evaluate the quality of the Denoising process, we use two quantitative metrics: Peak Signal-to-Noise Ratio (PSNR) and Structural Similarity Index (SSIM) \cite{wang2004image}. The PSNR is calculated as follows

\begin{equation*} PSNR := 10 \log_{10}\left(\frac{H \times W \times T \times ||\dot{\mathcal{X}}||_\infty^2}{|| \dot{\mathcal{Q}} - \dot{\mathcal{X}}||_F^2}\right), \end{equation*}
where $\dot{\mathcal{X}}$ represents the denoised output, $\dot{\mathcal{Q}}$ is the noisy input, and $|| \cdot ||_F$ denotes the Frobenius norm. Both metrics are widely used in such experiments, with higher values generally indicating better performance.
\\
The convergence criteria for Algorithm \ref{algo:solve_QTQ_problem} includes conditions on the Frobenius norm of the differences between successive iterations: $||\dot{\mathcal{Q}}^{(\tau+1)} - \dot{\mathcal{Q}}^{(\tau)}||_F^2 < tol$, $||\dot{\mathcal{S}}^{(\tau+1)} - \dot{\mathcal{S}}^{(\tau)}||_F^2 < tol$, and $||\dot{\mathcal{P}}^{(\tau+1)} - \dot{\mathcal{P}}^{(\tau)}||_F^2 < tol$ for Algorithm \ref{algo:solve_Tucker_Denoising}, where the tolerance $tol$ is set to $1 \times 10^{-6}$.

\subsection{Completion task}
In this section, we address the completion problem using data from \url{https://sbmi2015.na.icar.cnr.it/SBIdataset.html}. We consider the dataset Pedestrian. It is is reshaped to $48 \times 72$ and includes only 20 frame.\\
To demonstrate the effectiveness of the methods, we use the sample rate (\textbf{SR}), which indicates the percentage of missing pixels, chosen from multiple levels: ${0.1, 0.3, 0.5}$. A higher sample rate corresponds to a greater number of omitted pixels and, consequently, a more challenging task for the methods. The index set for the missing values is the same across the three channels but differs for each frame, which increases the difficulty of the problem.
\\
We found that the Geman function generally yields the best results; therefore, we will use it exclusively for our proposed method to avoid overloading the experiments. The $\gamma$ parameter is set to $3 \max(H, W)$, as suggested in \cite{cai2019tensor}. The maximum number of iterations is set to a low value (25), which is sufficient. We use $\rho = 1.1$, and $\beta^{\text{max}} = 10^4$.
\\
The methods that we compared are \textbf{Tucker, TTuckers, and TMAC}, using the parameters recommended by their authors. The results are presented in Table \ref{tab:result_Inpainting}.

\begin{table}[H]
\centering
\resizebox{\columnwidth}{!}{%
\begin{tabular}{|c|cc|cc|cc|cc|}
\hline
\multicolumn{1}{|c|}{\multirow{2}{*}{\textbf{Sample rate}}} & \multicolumn{2}{c|}{\textbf{LRQTC-NCTTR}} & \multicolumn{2}{c|}{\textbf{TTucker}} & \multicolumn{2}{c|}{\textbf{LRQTC-NCTR}} & \multicolumn{2}{c|}{\textbf{Tucker}} \\ \cline{2-9} 
\multicolumn{1}{|c|}{} & \multicolumn{1}{c|}{PSNR} & \multicolumn{1}{c|}{SSIM} & \multicolumn{1}{c|}{PSNR} & \multicolumn{1}{c|}{SSIM} & \multicolumn{1}{c|}{PSNR} & \multicolumn{1}{c|}{SSIM} & \multicolumn{1}{c|}{PSNR} & \multicolumn{1}{c|}{SSIM} \\ \hline
0.1 & 33.0441 & 992298 & 33.4868 & 0.992365 & \textbf{34.1270} & \textbf{0.993447} & 33.3255 & 0.992251 \\
0.3 & 28.2931 & 0.978310 & 27.1436 & 0.972419 & \textbf{28.3916} & \textbf{0.978666} & 26.3934 & 0.967918 \\
0.5 & 24.4501 & 0.951670 & 22.3719 & 0.929802 & \textbf{24.4827} & \textbf{0.953359} & 21.4833 & 0.916870 \\ \hline
\end{tabular}}
\caption{Inpaiting results with multiple Sample Rate on Pedestrain Dataset.}
\label{tab:result_Inpainting}
\end{table}

\noindent
The results reveal a notable advantage, particularly in scenarios where the sample rate is high, which indicates that the frames are more challenging due to a larger percentage of missing pixels. This higher difficulty level amplifies the effectiveness of the methods being evaluated, highlighting their robustness in handling more complex completion tasks.
\\
For a visual comparison of the methods, refer to Figure \ref{fig:Inpating_result}, which illustrates the performance of each approach. This figure provides a side-by-side view of the results, allowing for an intuitive assessment of how well each method performs in reconstructing the missing information under various conditions. The visual representation underscores the strengths and limitations of each technique, offering valuable insights into their relative effectiveness in dealing with high sample rates.
\begin{figure}[H]
\centering
\begin{subfigure}{0.13\textwidth}
\centering
\includegraphics[width=2.3cm]{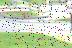}
\end{subfigure}
\quad
\begin{subfigure}{0.13\textwidth}
\centering
\includegraphics[width=2.3cm]{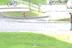}
\end{subfigure}
\quad
\begin{subfigure}{0.13\textwidth}
\centering
\includegraphics[width=2.3cm]{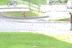}
\end{subfigure}
\quad
\begin{subfigure}{0.13\textwidth}
\centering
\includegraphics[width=2.3cm]{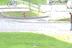}
\end{subfigure}
\quad
\begin{subfigure}{0.13\textwidth}
\centering
\includegraphics[width=2.3cm]{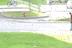}
\end{subfigure}
\quad
\begin{subfigure}{0.13\textwidth}
\centering
\includegraphics[width=2.3cm]{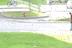}
\end{subfigure}
\quad
\begin{subfigure}{0.13\textwidth}
\centering
\includegraphics[width=2.3cm]{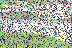}
\end{subfigure}
\quad
\begin{subfigure}{0.13\textwidth}
\centering
\includegraphics[width=2.3cm]{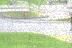}
\end{subfigure}
\quad
\begin{subfigure}{0.13\textwidth}
\centering
\includegraphics[width=2.3cm]{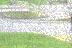}
\end{subfigure}
\quad
\begin{subfigure}{0.13\textwidth}
\centering
\includegraphics[width=2.3cm]{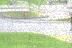}
\end{subfigure}
\quad
\begin{subfigure}{0.13\textwidth}
\centering
\includegraphics[width=2.3cm]{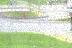}
\end{subfigure}
\quad
\begin{subfigure}{0.13\textwidth}
\centering
\includegraphics[width=2.3cm]{output_completion/pedestrian_10_restored_0.5_nc_Ttucker.jpeg}
\end{subfigure}
\caption{Visual results of different methods on different Sample rate $[0.1,0.5]$, on resized $48\times 72$ colored chosen frame image from Pedestrian Dataset. From left to right columns: Noised Frame, Tucker-nc, Ttucker, Ttucker-nc and TTmac.}
\label{fig:Inpating_result}
\end{figure}

\subsection{Denoising task}
In this section, we focus on the task of removing noise from the data. We utilize the Highway Dataset, it is reshaped to $92 \times 128$ and also includes only 10 frames. To evaluate the effectiveness of noise removal, we introduce varying levels of Gaussian noise, specifically $[0.1, 0.3, 0.5]$. For consistency across experiments, the subset of indices for missing values is identical across the three channels but varies from frame to frame.
\\
The parameter $\lambda$ is set to $1/\sqrt{(\max(H, W) \cdot T)}$, following the recommendation in \cite{lu2019tensor}. This parameter configuration is applied uniformly across all experiments. The maximum number of iterations is capped at 100, with $\rho$ set to 1.1, and $\beta^{\text{max}}$ to $10^4$.
\\
We employ both the Discrete Cosine Transform (DCT) and random orthogonal matrices in our proposed methods. These are denoted as \textbf{TRPCA-NC-dct} and \textbf{TRPCA-NC-rand}, respectively. Our results are compared against the baseline method, \textbf{TRPCA}.
\\
Figure \ref{fig:Denoise:_psnr_ssim} presents a detailed comparison of PSNR and SSIM metrics for each frame, highlighting the performance of different methods with a sample rat of 0.5.


\begin{figure}[H]
\centering
\begin{subfigure}{0.48\textwidth}
\includegraphics[width=0.99\linewidth]{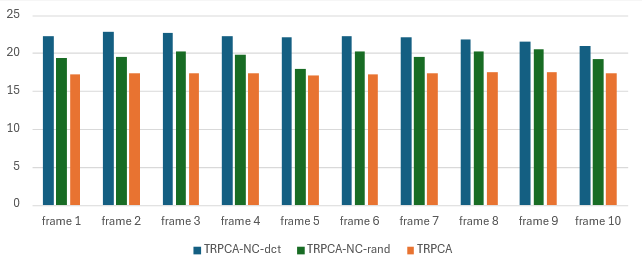}
\end{subfigure}
\quad
\begin{subfigure}{0.48\textwidth}
\includegraphics[width=0.99\linewidth]{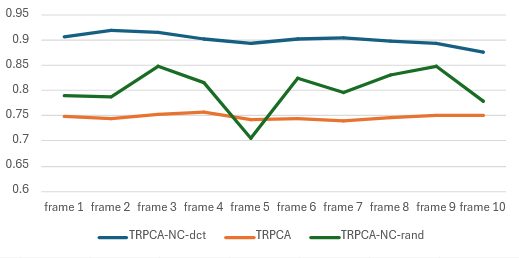}
\end{subfigure}
\caption{PSNR and SSIM results on different frames of Highway Dataset, with $SR=0.5$}
\label{fig:Denoise:_psnr_ssim}
\end{figure}
\noindent
The results clearly demonstrate that the proposed method utilizing the Discrete Cosine Transform (DCT) matrix outperforms other approaches in terms of Denoising efficacy. Specifically, the DCT-based method consistently delivers superior results across the majority of frames. In contrast, the method employing the random orthogonal matrix falls short in one particular frame, where it does not perform as well as the baseline method, \textbf{TRPCA}, in reducing noise.
\\
To provide a more comprehensive evaluation, Figure \ref{fig:Denoise} presents a visual comparison of the performance of each method. This figure illustrates how each technique handles the Denoising task, offering a side-by-side view that highlights the relative effectiveness of the DCT matrix compared to the orthogonal matrix and the \textbf{TRPCA} baseline. The visual representation underscores the strengths of the DCT-based approach and provides insight into the specific scenarios where the orthogonal matrix method may struggle.

\begin{figure}[H]
\centering
\begin{subfigure}{0.15\textwidth}
\centering
\includegraphics[width=2.5cm]{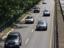}
\end{subfigure}
\quad
\begin{subfigure}{0.15\textwidth}
\centering
\includegraphics[width=2.5cm]{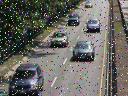}
\end{subfigure}
\quad
\begin{subfigure}{0.15\textwidth}
\centering
\includegraphics[width=2.5cm]{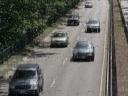}
\end{subfigure}
\quad
\begin{subfigure}{0.15\textwidth}
\centering
\includegraphics[width=2.5cm]{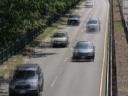}
\end{subfigure}
\quad
\begin{subfigure}{0.15\textwidth}
\centering
\includegraphics[width=2.5cm]{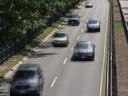}
\end{subfigure}
\quad
\begin{subfigure}{0.15\textwidth}
\centering
\includegraphics[width=2.5cm]{output_denoising/highway.jpeg}
\end{subfigure}
\quad
\begin{subfigure}{0.15\textwidth}
\centering
\includegraphics[width=2.5cm]{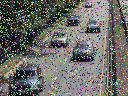}
\end{subfigure}
\quad
\begin{subfigure}{0.15\textwidth}
\centering
\includegraphics[width=2.5cm]{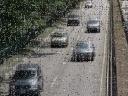}
\end{subfigure}
\quad
\begin{subfigure}{0.15\textwidth}
\centering
\includegraphics[width=2.5cm]{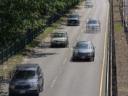}
\end{subfigure}
\quad
\begin{subfigure}{0.15\textwidth}
\centering
\includegraphics[width=2.5cm]{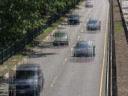}
\end{subfigure}
\quad
\begin{subfigure}{0.15\textwidth}
\centering
\includegraphics[width=2.5cm]{output_denoising/highway.jpeg}
\caption*{Original frame} 
\end{subfigure}
\quad
\begin{subfigure}{0.15\textwidth}
\centering
\includegraphics[width=2.5cm]{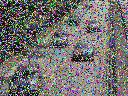}
\caption*{Noised} 
\end{subfigure}
\quad
\begin{subfigure}{0.15\textwidth}
\centering
\includegraphics[width=2.5cm]{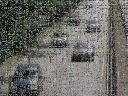}
\caption*{TRPCA}
\end{subfigure}
\quad
\begin{subfigure}{0.15\textwidth}
\centering
\includegraphics[width=2.5cm]{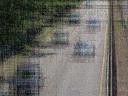}
\caption*{NC-rand}
\end{subfigure}
\quad
\begin{subfigure}{0.15\textwidth}
\centering
\includegraphics[width=2.5cm]{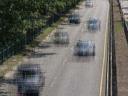}
\caption*{ NC-dct}
\end{subfigure}
\caption{Visual results of different methods on Noise level $\%$, on a frame image from Highway Dataset. The forth and fifth column, are TRPCA-NC-rand, and TRPCA-NC-dct, respectively.}
 \label{fig:Denoise}
\end{figure}
\noindent The visual results further corroborate the findings from the metrics, providing a compelling illustration of the method’s effectiveness. Specifically, the method \textbf{TRPCA-NC} demonstrates a marked improvement over state-of-the-art techniques. The visual comparisons reveal that \textbf{TRPCA-NC} not only achieves superior noise reduction but also preserves finer details and structural integrity of the images more effectively than its competitors.
\\
In the visual representations, \textbf{TRPCA-NC} consistently produces cleaner, more coherent images, with less visible noise and better overall clarity. The improved performance is evident in the more precise delineation of edges and textures, which are crucial for maintaining the quality of the denoised images. These observations align with the quantitative metrics, reinforcing the conclusion that \textbf{TRPCA-NC} significantly outperforms existing methods in the field.
\\
Overall, the combination of quantitative metrics and qualitative visual evidence underscores the robustness and superiority of the \textbf{TRPCA-NC} approach, validating its effectiveness in tackling noise reduction challenges compared to current state-of-the-art methods.

\section{Conclusion}\label{Sec:Conclusion}
The multi dimensional Data, especially the colored images sand videos make use of the tensor quaternion for a better representations, combining it with models that use Non-Convex surrogate functions to approximate the rank has shown to effective, compared to the original models that use the Nuclear norm, in both the Denoising and the Completion problems. We have develop several algorithms, along with the convergence analysis of these methods. As an extension, there is also the Octonions, which are double the dimension of the quaternions, can be used for analogous cases. Some steps can also be parallelized in both algorithms, to accelerate the convergence.

\section{Declaration of Competing Interest}
The authors declare that they have no known competing financial interests or personal relationships that could have appeared to influence the work reported in this paper.

\bibliographystyle{siam}
\bibliography{cas-refs}

\end{document}